\documentclass[12pt, a4paper]{amsart}
\usepackage[hmargin=30mm, vmargin=25mm, includefoot, twoside]{geometry}
\usepackage[bookmarksopen=true]{hyperref}
\usepackage[latin1]{inputenc}  % accents 8 bits dans la source
\usepackage[french,english]{babel}
\usepackage{amssymb,verbatim,  pstricks, pst-plot, textcomp, fontenc, euscript}
\usepackage[alphabetic]{amsrefs}
\usepackage{latexsym}
\usepackage{mathrsfs}
\usepackage{xspace}
\usepackage{enumerate, paralist}
\newtheorem{prop}{Proposition}%[section]

\newtheorem*{thm*}{Theorem}
\newtheorem*{prop*}{Proposition}

\newtheorem*{addendum*}{Addendum}
\newtheorem{cor}[prop]{Corollary}
\newtheorem{lem}[prop]{Lemma}

\newtheorem*{convention*}{Convention}
\theoremstyle{definition}
\newtheorem*{defn*}{Definition}

\newtheorem*{scholium*}{Scholium}
\theoremstyle{remark}

\newtheorem*{example*}{Example}
\numberwithin{equation}{section}

\newcommand{\vareps}{\varepsilon}

\newcommand{\RR}{\mathbf{R}}

\newcommand{\ZZ}{\mathbf{Z}}

\newcommand{\inv}{^{-1}}

\newcommand{\centra}{\mathscr{Z}}

\newcommand{\cat}{{\upshape CAT($0$)}\xspace}

\newcommand{\tangle}[2]% angle de Tits
{\angle_\mathrm{T}(#1,#2)}
\newcommand{\aangle}[3]% angle d'Alexandrov
{\angle_{#1}(#2,#3)}
\newcommand{\cangle}[3]% angle de comparaison
{\overline{\angle}_{#1}(#2,#3)}
\DeclareMathOperator{\rank}{rank}

\DeclareMathOperator{\Isom}{Is}

\def\rk{\mathrm{rank}}
\def\Min{\mathrm{Min}}

\begin{document}

\title{Regular elements in CAT(0) groups}

\author[P-E.~Caprace]{Pierre-Emmanuel Caprace{$^*$}}
\address{Universit\'e catholique de Louvain, IRMP, Chemin du Cyclotron 2, 1348 Louvain-la-Neuve, Belgium}
\email{pe.caprace@uclouvain.be}
\thanks{{$^*$} F.R.S.-FNRS Research Associate, supported in part by FNRS grant F.4520.11 and the European Research Council}

\author[G.~Zadnik]{Ga\v{s}per Zadnik{$^\star$}}
\address{In\v{s}titut za matematiko, fiziko in mehaniko, Jadranska ulica 19, SI-1111 Ljubljana, Slovenia}
\email{zadnik@fmf.uni-lj.si}
\thanks{$^\star$ Supported by the Slovenian Research Agency and in part by
the Slovene Human Resources Development and Scholarship Fund}

\date{December  2011}
%\keywords{}

%%%%%%%%%%%%%%%%%%%%%%%%%%%%%%%%%%%%%%%%%%%%%%%%%%%%%
\maketitle
%%%%%%%%%%%%%%%%%%%%%%%%%%%%%%%%%%%%%%%%%%%%%%%%%%%%%
\begin{abstract}
Let $X$ be a locally compact geodesically complete CAT(0) space and $\Gamma$ be a discrete group acting properly and cocompactly on $X$. We show that $\Gamma$ contains an element acting as a hyperbolic isometry on each indecomposable de Rham factor of $X$. It follows that if $X$ is a product of $d$ factors, then $\Gamma$ contains $\ZZ^d$. 
\end{abstract}
%%%%%%%%%%%%%%%%%%%%%%%%%%%%%%%%%%%%%%%%%%%%%%%%%%%%%

\bigskip

Let $X$ be a proper \cat space and $\Gamma$ be a discrete group acting properly and cocompactly by isometries on $X$. The \emph{flat closing conjecture} predicts that if $X$ contains a $d$-dimensional flat, then $\Gamma$ contains a copy of $\ZZ^d$ (see \cite[Section~6.B$_3$]{Gro}). In the special case $d=2$, this would imply that $\Gamma$ is hyperbolic if and only if it does not contain a copy of $\ZZ^2$. This notorious conjecture remains however open as of today. It holds  when $X$ is a real analytic manifold of non-positive sectional curvature by the main result of \cite{BS}. In the classical case when $X$ is a non-positively curved symmetric space, it can be established with the following simpler and well known argument: by \cite[Appendix]{BL}, the group $\Gamma$ must contain a so called \textbf{$\RR$-regular} semisimple element, \emph{i.e.} a hyperbolic isometry $\gamma$ whose axes are contained in a unique maximal flat of $X$. By a lemma of Selberg~\cite{Sel}, the centraliser $\centra_\Gamma(\gamma)$ is a lattice in the centraliser $\centra_{\Isom(X)}(\gamma)$. Since the latter centraliser is virtually $\RR^d$ with $d = \rk(X)$, one concludes that $\Gamma$ contains $\ZZ^d$, as desired. 

It is tempting to try and mimick that strategy of proof in the case of a general \cat space $X$: if one shows that $\Gamma$ contains a hyperbolic isometry $\gamma$ which is \textbf{maximally regular} in the sense that its axes are contained in a unique flat of maximal possible dimension among all flats of $X$, then the flat closing conjecture will follow as above. The main result of this note provides hyperbolic isometries satisfying a weaker notion of regularity. 

\begin{thm*}\label{regular}
Assume that $X$ is geodesically complete. 

Then $\Gamma$ contains a hyperbolic element which acts as a hyperbolic isometry on each indecomposable de Rham factor of $X$. 
\end{thm*}

Every \cat space $X$ as in the theorem admits a canonical de Rham decomposition, see \cite[Corollary~5.3(ii)]{CMST}. Notice that the number of indecomposable de Rham factors of $X$ is a lower bound on the dimension of all maximal flats in $X$, although two such maximal flats need not have the same dimension in general. As expected, we deduce a corresponding lower bound on the maximal rank of free abelian subgroups of $\Gamma$. 

\begin{cor}\label{cor:Z^d}
If $X$ is a product of $d$  factors, then $\Gamma$ contains a copy of $\ZZ^d$. 
\end{cor}

We believe that those results should hold without the assumption of geodesic completeness; in case $X$ is a \cat cube complex, this is indeed so, see \cite[\S 1.3]{CaSa}.

\medskip
The proof of the theorem and its corollary relies in an essential way on results from \cite{CMST} and \cite{CMDS}. The first step consists in applying \cite[Theorem~1.1]{CMST}, which ensures that $X$ splits as 
$$
X \cong \RR^d \times M \times Y_1 \times \dots \times Y_q,
$$ 
where $M$ is a symmetric space of non-compact type and the factors $Y_i$ are geodesically complete indecomposable \cat spaces whose full isometry group is totally disconnected. Moreover this decomposition is canonical, hence preserved by a finite index subgroup of $\Isom(X)$ (and thus of $\Gamma$). The next essential point is that, by  \cite[Theorem 3.8]{CMDS}, the group $\Gamma$ virtually splits as $\ZZ^d \times \Gamma'$, and the factor $\Gamma'$ (resp. $\ZZ^d$) acts properly and cocompactly on $M \times Y_1 \times \dots \times Y_q$ (resp. $\RR^d$). Therefore, our main theorem  is a consequence of the following. 

\begin{prop}\label{main}
Let $X = M \times Y_1 \times \dots \times Y_q$, where $M$ is a symmetric space of non-compact type and $Y_i$ is a geodesically complete locally compact \cat space with totally disconnected isometry group. 
%Let $G \leq \Isom(X)$ be a closed subgroup acting cocompactly on $X$.  

Any discrete cocompact group of isometries of $X$  contains an element acting  as an $\RR$-regular hyperbolic element on $M$, and as a hyperbolic element on $Y_i$ for all $i$. 
\end{prop}

As before, this yields a lower bound on the rank of maximal free abelian subgroups of $\Gamma$, from which Corollary~\ref{cor:Z^d} follows. 
 
%As a consequence, we obtain the following result,  from which Corollary~\ref{cor:Z^d} follows. 

\begin{cor}\label{cor:main}
Let $X = M \times Y_1 \times \dots \times Y_q$ be as in the proposition. 
Then any discrete cocompact group of isometries of $X$ contains a copy of $\ZZ^{\rk(M) + q}$. 
\end{cor}

\begin{proof}%[Proof of Corollary~\ref{cor:main}]
Let $\Gamma < \Isom(X)$ be a discrete subgroup acting cocompactly. Upon replacing $\Gamma$ by a subgroup of finite index, we may assume that $\Gamma$ preserves the given product decomposition of $X$ (see \cite[Corollary~5.3(ii)]{CMST}).  Let $\gamma \in \Gamma$ be as in Proposition~\ref{main} and let $\gamma_M$ (resp. $\gamma_i$) be its projection to $\Isom(M)$ (resp. $\Isom(Y_i)$). Then $\Min(\gamma_M) = \RR^{\rank(M)}$ and for all $i$ we have $\Min(\gamma_i) \cong \RR \times C_i$ for some \cat space $C_i$, by \cite[Theorem~II.6.8(5)]{BH}. Hence the desired conclusion follows from the following lemma. 
\end{proof}

%To deduce the corollary from the proposition, we use the following. 

\begin{lem}
Let $X = X_1 \times \dots \times X_p$ be a proper \cat space and $\Gamma$ a discrete group acting properly cocompactly on $X$. Let also $\gamma \in \Gamma$ be an element preserving some $d_i$-dimensional flat in $X_i$ on which it acts by translation, for all $i$. 

Then $\Gamma$ contains a free abelian group of rank  $d_1 + \dots + d_p$. 
\end{lem}

\begin{proof}
%Consider the set $\Min(\gamma)$ consisting of all those points where the displacement function of $\gamma$ achieves its infimum. 
By assumption $\gamma$ preserves the given product decomposition of $X$. We let   $\gamma_i$ denote the projection of $\gamma$ on  $\Isom(X_i)$.  Observe that 
$$\Min(\gamma) = \Min(\gamma_1) \times \dots \times \Min(\gamma_p).$$ 
 By hypothesis, we have $\Min(\gamma_i) \cong \RR^{d_i} \times C_i$ for some \cat space $C_i$. 
%, by \cite[Theorem~II.6.8(5)]{BH}. 
Therefore $\Min(\gamma) \cong \RR^{d_1 + \dots + d_p} \times C_1 \times \dots \times C_p$. By~\cite[Theorem 3.2]{Rua} the centraliser $\centra_\Gamma(\gamma)$ acts cocompactly (and of course properly) on $\Min(\gamma)$. Therefore, invoking \cite[Theorem 3.8]{CMDS}, we infer  that   $\ZZ^{d_1 + \dots + d_p}$ is a (virtual) direct factor of $\centra_\Gamma(\gamma)$. 
\end{proof}

%They follow respectively from the Theorem and Proposition~\ref{main}, combined with the following. 

\medskip

It remains to prove Proposition~\ref{main}. We  proceed in three steps. The first one provides an element $\gamma_Y \in \Gamma$ acting as a hyperbolic isometry on each $Y_i$. This combines an argument of E.~Swenson~\cite[Theorem 11]{Swe} with the phenomenon of \textbf{Alexandrov angle rigidity}, described in \cite[Proposition~6.8]{CMST} and recalled below. The latter requires the hypothesis of geodesic completeness. The second step uses that $\Gamma$ has  subgroups acting properly cocompactly on $M$, and thus contains an element $\gamma_M$ acting as an $\RR$-regular isometry of $M$ by \cite{BL}. The last step uses a result from \cite{PR} ensuring that for all elements $\delta'$ in some Zariski open subset of $\Isom(M)$ and all sufficiently large $n>0$, the product  $\gamma_M^n \delta'$ is $\RR$-regular. Invoking the Borel density theorem, we finally find an appropriate element $\delta \in \Gamma$ such that the product $\gamma = \gamma_M^n \delta \gamma_Y$ has the requested properties. We now proceed to the details. %, retaining the notation and assumptions of Proposition~\ref{main}.

\begin{prop*}[\textbf{Alexandrov angle rigidity}]
Let $Y$ be a locally compact geodesically complete \cat space and $G$ be a totally disconnected locally compact group acting continuously, properly and cocompactly on $Y$ by isometries. 

Then there is $\vareps >0$ such that for any elliptic isometry $g \in G$ and any $x \in X$ not fixed by $g$, we have $\aangle c {gx} x \geq \vareps$, where $c$ denotes the projection of $x$ on the set of $g$-fixed points.
\end{prop*}
\begin{proof}
See \cite[Proposition~6.8]{CMST}.
\end{proof}

\begin{prop}\label{prop:TodDisc}
Let $Y = Y_1 \times \dots \times Y_q$, where $Y_i$ is a geodesically complete locally compact \cat space with totally disconnected isometry group, and $G$ be a locally compact group acting continuously, properly and cocompactly by isometries on $Y$. 

Then $G$ contains an element acting on $Y_i$ as a hyperbolic isometry for all $i$.
\end{prop}

\begin{proof}
Upon replacing $G$ by a finite index subgroup, we may assume that $G$ preserves the given product decomposition of $Y$, see \cite[Corollary~5.3(ii)]{CMST}. Let $\rho : [0, \infty) \to Y$ be a geodesic ray which is \textbf{regular}, in the sense that its projection to each $Y_i$ is a ray (in other words the end point $\rho(\infty)$ does not belong to the boundary of a subproduct). 

Since $G$ is cocompact, we can find a sequence $(g_n)$  in $G$ and a sequence $(t_n)$   in $\RR_+$ such that   $g_n.\rho(t_n)$ converges to some point $y \in Y$ and $g_n.\rho$ converges uniformly on compacta to a geodesic line $\ell$ in $Y$. Set $h_{i,j} =g_i^{-1}g_j \in G$ and consider the angle
$$\theta =\angle_{\rho(t_i)}(h_{i,j}^{-1}.\rho(t_i),h_{i,j}.\rho(t_i)).$$ 
As in \cite[Theorem 11]{Swe}, observe that $\theta$ is arbitrarily close to $\pi$ for $i<j$ large enough. 
%Alexandrov angle rigidity implies that for all $i<j$ large enough, the isometry $h_{i, j}$ is hyperbolic, see  \cite[Corollary 6.9]{CMST} (in the case when $G$ is discrete, this is precisely  Swenson's argument from \cite[Theorem 11]{Swe}). In particular the proposition holds in case  $q=1$. 
%Assume $q\geq 2$. 

We shall prove that for all $i<j$ large enough, the isometry $h_{i, j}$ is regular hyperbolic, in the sense that its projection to each factor $Y_k$ is hyperbolic. We argue by contradiction and assume that this is not the case. Notice that $\Isom(Y_k)$ does not contain any parabolic isometry by~\cite[Corollary~6.3(iii)]{CMST}. Therefore, upon extracting and reordering the factors, we may then assume that there is some $s \leq q$ such that for all $i<j$, the projection of $h_{i, j}$  on $\Isom(Y_1), \dots, \Isom(Y_s)$ is elliptic, and the projection of $ h_{i, j}$  on $\Isom(Y_{s+1}), \dots, \Isom(Y_q)$ is hyperbolic. 
We set $Y' = Y_1 \times \dots \times Y_s$ and $Y'' = Y_{s+1} \times \dots \times Y_q$. We shall prove that for $i<j$ large enough, the projections of $(h_{i, j})$ on $\Isom(Y')$ forms a sequence of elliptic isometries which contradict Alexandrov angle rigidity. 

Fix some small $\delta >0$. Let $x_i$  (resp. $y_i$) be the point at distance $\delta$ from $\rho(t_i)$ and lying on the geodesic segment   $[h_{i,j}^{-1}.\rho(t_i),\rho(t_i)]$ (resp. $[\rho(t_i),h_{i,j}.\rho(t_i)]$). By construction, for $i<j$ large enough, the union of the two geodesic segments $[x_i, \rho(t_i)] \cup [\rho(t_i), y_i]$ lies in an arbitrary small tubular neighbourhood of the geodesic ray $\rho$. 
Since the projection $Y \to Y'$ is $1$-Lipschitz, it follows that the $Y'$-component of $[x_i, \rho(t_i)] \cup [\rho(t_i), y_i]$, which we denote by $[x'_i, \rho'(t_i)] \cup [\rho'(t_i), y'_i]$, is uniformly close to the $Y'$-component of $\rho$, say $\rho'$. Since $\rho$ is a regular ray, its projection $\rho'$ is also a geodesic ray. Therefore, the angle 
$$
\theta' =\angle_{\rho'(t_i)}(x'_i,y'_i)
$$
is arbitrarily close to $\pi$ for $i < j$ large enough. Pick $i <j$ so large that $\theta' > \pi - \vareps $, where $\vareps > 0$ is the constant from Alexandrov angle rigidity for $Y'$. Set $h = h_{i, j}$ and let $h'$ be the projection of $h$ on $\Isom(Y')$. By assumption $h'$ is elliptic. Let $c$ denote the projection of $\rho'(t_i)$ on the set of $h'$-fixed points. Then the isosceles triangles $\triangle \big(c, (h')^{-1}.\rho'(t_i),\rho'(t_i)\big)$ and $\triangle \big(c, \rho'(t_i),h'.\rho'(t_i)\big)$ are congruent, and we deduce   
$$
\begin{array}{rcl}
\angle_{c}(\rho'(t_i), h'.\rho'(t_i)) &\leq  &\pi - \aangle {\rho'(t_i)} c {h'.\rho'(t_i)} - \aangle {\rho'(t_i)} c {(h')\inv.\rho'(t_i)}\\
&  \leq & \pi - \angle_{\rho'(t_i)}((h')^{-1}.\rho'(t_i),h'.\rho'(t_i))\\
& = & \pi - \theta'\\
&  <&\vareps.
\end{array}
$$ 
This contradicts Alexandrov angle rigidity.
\end{proof}

\begin{proof}[Proof of Proposition~\ref{main}]
Let $\Gamma$ be a discrete group acting properly and cocompactly on $X$. 
First observe that (after passing to a finite index subgroup) we may  assume that $\Gamma$ preserves  the given product decomposition of $X$, see see \cite[Corollary~5.3(ii)]{CMST}. 

Let $G$ be the closure of the projection of $\Gamma$ to $\Isom(Y_1) \times \dots \times \Isom(Y_q)$. Then $G$ acts properly cocompactly on $Y = Y_1 \times \dots \times Y_q$. Therefore it contains an element $g$ acting as a hyperbolic isometry on $Y_i$ for all $i$ by Proposition~\ref{prop:TodDisc}. Since $\Gamma$ maps densely to $G$ and since the stabiliser of each point of $Y$ in $G$ is open by \cite[Theorem~1.2]{CMST}, it follows that $\Gamma$-orbits on $Y \times Y$ coincide with the $G$-orbits. In particular, given $y \in \Min(g)$, we can find $\gamma_Y \in \Gamma$ such that $\gamma_Y(y, g\inv y) = (gy, y)$. Since $\aangle y {\gamma_Y\inv y} {\gamma_Y y} = \aangle y {g\inv y} {g y} = \pi$, we infer that $\gamma_Y$ is hyperbolic and has an axis containing the segment $[g\inv y, gy]$. In particular  $\gamma_Y $ acts as a hyperbolic isometry on $Y_i$ for all $i$. 

Let $\gamma_Y = (\alpha, h) $ be the decomposition of $\gamma_Y$ along the splitting $\Isom(X) = \Isom(M) \times \Isom(Y)$.  By construction $h$ acts as a hyperbolic isometry on $Y_i$ for all $i$. 

Let $U \leq \Isom(Y)$ be the pointwise stabiliser of a ball containing $y, \gamma_Y y $ and $\gamma_Y\inv y$. Notice that every element of $\Isom(Y)$ contained in the coset $Uh$ maps $y$ to $h.y$ and $h\inv y$ to $y$, and therefore acts also  as a hyperbolic isometry on $Y_i$ for all $i$. 

On the other hand  $U$ is a compact open subgroup of $\Isom(Y)$ by \cite[Theorem 1.2]{CMST}. Set $\Gamma_U = \Gamma \cap (\Isom(M)\times U)$. Notice that $\Gamma_U$ acts properly and cocompactly on $M$ by \cite[Lemma 3.2]{CMDS}. In other words the projection of $\Gamma_U$ to $\Isom(M)$ is a cocompact lattice. Abusing notation slightly, we shall denote this projection equally by $\Gamma_U$. 

By the appendix from  \cite{BL} (see also  \cite{P} for an alternative argument), the group $\Gamma_U$ contains an element $\gamma_M$ acting as an $\RR$-regular element on $M$.  
By \cite[Lemma 3.5]{PR} there is a Zariski open set $V = V(\gamma_M)$ in $\Isom(M)$ with the following property. For any $\delta \in V$ there exists $n_\delta$ such that an element $\gamma_M^n \delta$ is $\RR$-regular for any $n\geq n_\delta$. By the Borel density theorem, the intersection $\Gamma_U \cap V\alpha^{-1}$ is nonempty.  Pick an element $\delta \in \Gamma_U \cap V\alpha^{-1}$. Then $\delta \alpha\in V$ which means by definition that $\gamma_M^n \delta \alpha$ is $\RR$-regular for all $n\geq n_0$ for some integer $n_0$. 

Pick an element $\gamma'_M \in \Gamma$ (resp. $\delta' \in \Gamma$) which lifts $\gamma_M$ (resp. $\delta$). Set 
$$\gamma = (\gamma'_M)^{n_0} \delta'  \gamma_Y \in \Gamma_U.
$$ 
The projection of $\gamma$ to $\Isom(M)$ is $\gamma_M^{n_0} \delta \alpha$ and is thus $\RR$-regular. The projection of $\gamma$ to $\Isom(Y)$ belongs to the coset $Uh$, and therefore acts as a hyperbolic isometry on $Y_i$ for all $i$. 
\end{proof}


\begin{bibdiv}
\begin{biblist}
\bib{BS}{article}{
   author={Bangert, Victor},
   author={Schroeder, Viktor},
   title={Existence of flat tori in analytic manifolds of nonpositive
   curvature},
   journal={Ann. Sci. \'Ecole Norm. Sup. (4)},
   volume={24},
   date={1991},
   number={5},
   pages={605--634},
   issn={0012-9593},
%   review={\MR{1132759 (92k:53110)}},
}
\bib{BH}{book}{
   author={Bridson, Martin R.},
   author={Haefliger, Andr{\'e}},
   title={Metric spaces of non-positive curvature},
   series={Grundlehren der Mathematischen Wissenschaften [Fundamental
   Principles of Mathematical Sciences]},
   volume={319},
   publisher={Springer-Verlag},
   place={Berlin},
   date={1999},
   pages={xxii+643},
%   isbn={3-540-64324-9},
%   review={\MR{1744486 (2000k:53038)}},
}
\bib{BL}{article}{
   author={Benoist, Yves},
   author={Labourie, Fran{\c{c}}ois},
   title={Sur les diff\'eomorphismes d'Anosov affines \`a\ feuilletages
   stable et instable diff\'erentiables},
   language={French, with French summary},
   journal={Invent. Math.},
   volume={111},
   date={1993},
   number={2},
   pages={285--308},
   issn={0020-9910},
%   review={\MR{1198811 (94d:58114)}},
 %  doi={10.1007/BF01231289},
}
\bib{CaSa}{article}{
   author={Caprace, Pierre-Emmanuel},
   author={Sageev, Michah},
   title={Rank rigidity for {CAT}(0) cube complexes},
   journal={Geom. Funct. Anal.},
   volume={21},
   date={2011},
   number={4},
   pages={851--891},
}
\bib{CMST}{article}{
   author={Caprace, Pierre-Emmanuel},
   author={Monod, Nicolas},
   title={Isometry groups of non-positively curved spaces: structure theory},
   journal={J. Topol.},
   volume={2},
   date={2009},
   number={4},
   pages={661--700},
   issn={1753-8416},
%   review={\MR{2574740 (2011i:53051)}},
 %  doi={10.1112/jtopol/jtp026},
}
\bib{CMDS}{article}{
   author={Caprace, Pierre-Emmanuel},
   author={Monod, Nicolas},
   title={Isometry groups of non-positively curved spaces: discrete
   subgroups},
   journal={J. Topol.},
   volume={2},
   date={2009},
   number={4},
   pages={701--746},
   issn={1753-8416},
%   review={\MR{2574741 (2011i:53052)}},
 %  doi={10.1112/jtopol/jtp027},
}
\bib{Gro}{article}{
   author={Gromov, Mikhail},
   title={Asymptotic invariants of infinite groups},
   conference={
      title={Geometric group theory, Vol.\ 2},
      address={Sussex},
      date={1991},
   },
   book={
      series={London Math. Soc. Lecture Note Ser.},
      volume={182},
      publisher={Cambridge Univ. Press},
      place={Cambridge},
   },
   date={1993},
   pages={1--295},
%   review={\MR{1253544 (95m:20041)}},
}
\bib{K}{article}{
   author={Kleiner, Bruce},
   title={The local structure of length spaces with curvature bounded above},
   journal={Math. Z.},
   volume={231},
   date={1999},
   number={3},
   pages={409--456},
   issn={0025-5874},
%   review={\MR{1704987 (2000m:53053)}},
 %  doi={10.1007/PL00004738},
}
\bib{P}{article}{
   author={Prasad, Gopal},
   title={${\bf R}$-regular elements in Zariski-dense subgroups},
   journal={Quart. J. Math. Oxford Ser. (2)},
   volume={45},
   date={1994},
   number={180},
   pages={541--545},
   issn={0033-5606},
   review={\MR{1315463 (96a:22022)}},
}
\bib{PR}{article}{
   author={Prasad, Gopal},
   author={Raghunathan, Madabusi Santanam},
   title={Cartan subgroups and lattices in semi-simple groups},
   journal={Ann. of Math. (2)},
   volume={96},
   date={1972},
   pages={296--317},
   issn={0003-486X},
%   review={\MR{0302822 (46 \#1965)}},
}
\bib{Rua}{article}{
   author={Ruane, Kim E.},
   title={Dynamics of the action of a ${\rm CAT}(0)$ group on the boundary},
   journal={Geom. Dedicata},
   volume={84},
   date={2001},
   number={1-3},
   pages={81--99},
   issn={0046-5755},
%   review={\MR{1825346 (2002d:20064)}},
 %  doi={10.1023/A:1010301824765},
}
\bib{Sel}{article}{
   author={Selberg, Atle},
   title={On discontinuous groups in higher-dimensional symmetric spaces},
   conference={
      title={Contributions to function theory (internat. Colloq. Function
      Theory, Bombay, 1960)},
   },
   book={
      publisher={Tata Institute of Fundamental Research},
      place={Bombay},
   },
   date={1960},
   pages={147--164},
%   review={\MR{0130324 (24 \#A188)}},
}
\bib{Swe}{article}{
   author={Swenson, Eric L.},
   title={A cut point theorem for ${\rm CAT}(0)$ groups},
   journal={J. Differential Geom.},
   volume={53},
   date={1999},
   number={2},
   pages={327--358},
   issn={0022-040X},
%   review={\MR{1802725 (2001i:20083)}},
}
\end{biblist}
\end{bibdiv}
\end{document}